\documentclass{article}

\usepackage{amssymb}
\usepackage{amsmath}
\usepackage{amsthm}
\usepackage{tikz-cd}
\usepackage[top=1 in, right=1 in, left=1 in, bottom=1 in]{geometry}
\usepackage{graphicx}
\graphicspath{ {./graphics_for_paper/} }
\usepackage{wrapfig}
\usepackage{caption}
\usepackage{subcaption}
\usepackage{enumitem}

\usepackage{algorithm}
\usepackage{algpseudocode}
\usepackage{authblk}
\usepackage{mathrsfs}
\usepackage{bigints}
\usepackage{relsize}

\usepackage{mathabx}

\usepackage{lipsum}

\usepackage[square,sort,comma]{natbib}

\usepackage{color}
\definecolor{MyDarkRed}{rgb}{0.5,0,0.1}
\definecolor{MyDarkBlue}{rgb}{0.1,0.1,0.5}
\definecolor{MyDarkGreen}{rgb}{0.1,0.5,0.1}
\definecolor{MyRed}{rgb}{1.0,0,0}
\definecolor{MyBlue}{rgb}{0,0,1.0}
\definecolor{MyGreen}{rgb}{0,0.8,0}
\definecolor{lightgray}{rgb}{0.96,0.96,0.96}
\definecolor{darkgray}{rgb}{0.4,0.4,0.4}
\definecolor{darkergray}{rgb}{0.25,0.25,0.25}
\usepackage[pdftex, colorlinks=true, urlcolor=MyDarkBlue, citecolor=MyDarkRed, linkcolor=MyDarkGreen, hyperfootnotes=false]{hyperref}

\newtheorem{definition}{Definition}
\newtheorem{theorem}{Theorem}[section]
\newtheorem{proposition}[theorem]{Proposition}
\newtheorem{lemma}[theorem]{Lemma}
\newtheorem{corollary}[theorem]{Corollary}

\newtheorem*{proposition*}{Proposition}

\theoremstyle{definition}
\newtheorem{example}{Example}[section]

\usepackage[framemethod=TikZ]{mdframed}

\newcommand\Item[1][]{%
	\ifx\relax#1\relax  \item \else \item[#1] \fi
	\abovedisplayskip=0pt\abovedisplayshortskip=0pt~\vspace*{-\baselineskip}}

\usepackage{etoolbox}

\mdfsetup{%
	middlelinewidth=2pt,
	roundcorner=10pt}

\definecolor{MyLightGray8}{rgb}{0.8,0.8,0.8} %{1,1,1} %{0.8,0.8,0.8}
\definecolor{MyLightGray9}{rgb}{0.9,0.9,0.9} %{1,1,1} %{0.8,0.8,0.8}
\definecolor{MyLightGray95}{rgb}{0.95,0.95,0.95} %{1,1,1} %{0.95,0.95,0.95}

\definecolor{proofcolor}{rgb}{0.96,0.96,0.96}

\BeforeBeginEnvironment{proof}{%
	\begin{mdframed}[topline=false,
				rightline=false,
				bottomline=false,
				leftline=false,
				innerleftmargin=1ex,
				innerrightmargin=1ex,
				innertopmargin=1ex,
				innerbottommargin=1ex,
				backgroundcolor=proofcolor,
				]\hyphenchar\font=\defaulthyphenchar %\RaggedRight
}
\AfterEndEnvironment{proof}{%
\end{mdframed}\hyphenchar\font=-1
}

\definecolor{lemmacolor}{rgb}{1.0,0.95,0.9}

\BeforeBeginEnvironment{lemma}{%
	\begin{mdframed}[topline=false,
		rightline=false,
		bottomline=false,
		leftline=false,
		innerleftmargin=1ex,
		innerrightmargin=1ex,
		innertopmargin=1ex,
		innerbottommargin=1ex,
		backgroundcolor=lemmacolor,
		]\hyphenchar\font=\defaulthyphenchar %\RaggedRight
	}
	\AfterEndEnvironment{lemma}{%
	\end{mdframed}\hyphenchar\font=-1
}

\definecolor{corcolor}{rgb}{0.95,0.95,0.9}

\BeforeBeginEnvironment{corollary}{%
	\begin{mdframed}[topline=false,
		rightline=false,
		bottomline=false,
		leftline=false,
		innerleftmargin=1ex,
		innerrightmargin=1ex,
		innertopmargin=1ex,
		innerbottommargin=1ex,
		backgroundcolor=corcolor,
		]\hyphenchar\font=\defaulthyphenchar %\RaggedRight
	}
	\AfterEndEnvironment{corollary}{%
	\end{mdframed}\hyphenchar\font=-1
}

\definecolor{propcolor}{rgb}{1.0,0.93,0.94}

\BeforeBeginEnvironment{proposition}{%
	\begin{mdframed}[topline=false,
		rightline=false,
		bottomline=false,
		leftline=false,
		innerleftmargin=1ex,
		innerrightmargin=1ex,
		innertopmargin=1ex,
		innerbottommargin=1ex,
		backgroundcolor=propcolor,
		]\hyphenchar\font=\defaulthyphenchar %\RaggedRight
	}
	\AfterEndEnvironment{proposition}{%
	\end{mdframed}\hyphenchar\font=-1
}

\definecolor{thmcolor}{rgb}{1.0,0.90,0.91}

\BeforeBeginEnvironment{theorem}{%
	\begin{mdframed}[topline=false,
		rightline=false,
		bottomline=false,
		leftline=false,
		innerleftmargin=1ex,
		innerrightmargin=1ex,
		innertopmargin=1ex,
		innerbottommargin=1ex,
		backgroundcolor=thmcolor,
		]\hyphenchar\font=\defaulthyphenchar %\RaggedRight
	}
	\AfterEndEnvironment{theorem}{%
	\end{mdframed}\hyphenchar\font=-1
}

\definecolor{defcolor}{rgb}{0.92,0.92,0.95}

\BeforeBeginEnvironment{definition}{%
	\begin{mdframed}[topline=false,
		rightline=false,
		bottomline=false,
		leftline=false,
		innerleftmargin=1ex,
		innerrightmargin=1ex,
		innertopmargin=1ex,
		innerbottommargin=1ex,
		backgroundcolor=defcolor,
		]\hyphenchar\font=\defaulthyphenchar %\RaggedRight
	}
	\AfterEndEnvironment{definition}{%
	\end{mdframed}\hyphenchar\font=-1
}

\newcommand{\tr}[0]{\mathrm{Tr}}
\newcommand{\diag}[0]{\mathrm{diag}}

\newcommand{\cgeq}{\succcurlyeq}
\newcommand{\cleq}{\preccurlyeq}

\title{Trace of Multi-variable Matrix Functions \\ and its Application to Functions of Graph Spectrum}
\author{Subhrajit Bhattacharya\footnote{
		Lehigh University, Bethlehem, PA, U.S.A. E-mail: \href{mailto:sub216@lehigh.edu}{\texttt{sub216@lehigh.edu}}.\\
		This work was supported by AFOSR award number FA9550-23-1-0046.
	}}
\date{}

\begin{document}

\maketitle

\begin{abstract}\noindent
	Matrix extension of a scalar function of a single variable is well-studied in literature. Of particular interest is the trace of such functions. It is known that for diagonalizable matrices, $M$, the function $g(M) = \text{Tr}(f(M)) = \sum_{j=1}^n f(\mu_j)$ (where $\{\mu_j\}_{j=1,2,\cdots,n}$ are the eigenvalues of $M$) inherits the monotonocity and convexity properties of $f$ (\emph{i.e.}, for $g$ to be convex, $f$ need not be operator convex -- convexity is sufficient).
	In this paper we formalize the idea of matrix extension of a function of multiple variables, study the monotonicity and convexity properties of the trace, and thus show that a function of form $g(M) = \sum_{j_1=1}^n \sum_{j_2=1}^n \cdots \sum_{j_m=1}^n f(\mu_{j_1}, \mu_{j_2},\cdots, \mu_{j_m})$ also inherits the monotonocity and convexity properties of the multi-variable function, $f$.
	We apply these results to functions of the spectrum of the weighted Laplacian matrix of undirected, simple graphs.
\end{abstract}

\section{Introduction: Matrix Extension of a Scalar Function of a Single Variable}

In this section we present known results on matrix extension of a scalar function of a single variable.
We start with a few definitions.

\begin{definition}
	Let $\mathcal{M}_{\text{diag}^n}$ be the set of $n\times n$ diagonalizable matrices, $\mathcal{M}_{\text{norm}^n}$ be the set of $n\times n$ normal matrices, and $\mathcal{M}_{\text{herm}^n}$ be the set of $n\times n$ Hermitian matrices. Furthermore, given $S\subseteq \mathbb{C}$ (where $S$ can possibly be a subset of the real line), we define $\mathcal{M}_{\text{typ}^n}(S)$ to be the set of $n\times n$ matrices of type `$\text{typ}$' whose eigenvalues lie in $S$. %, where $\mathbb{K}$ is either $\mathbb{R}$
\end{definition}

\begin{definition}[Partial Order based on Positive Semi-definiteness of Difference]
	For matrices $A,B$, we say $A \cgeq B$ (equivalently, $B \cleq A$) if $A-B$ is positive semi-definite (\emph{i.e.}, $\mathbf{x}^* (A-B) \mathbf{x} \geq 0, ~\forall \mathbf{x} \in \mathbb{C}^n$).
\end{definition}

\vspace{1em}\noindent
Evaluating a scalar function of a single variable at a matrix (also called extension of a scalar function to a matrix function) is well-studied in literature. %In this section we present some well-known results.

\begin{definition}[Extension of a Scalar Function to Diagonalizable Matrices~\citep{bhatia}]\label{def:function-single-var}
Consider a function $f:S \rightarrow \mathbb{C}$ (where $S$ is a subset of $\mathbb{C}$, with the possibility that $S \subseteq \mathbb{R}$ and/or it's a real-valued function). Given a $n\times n$ diagonalizable matrix, $M = P \,\diag(\lambda_1, \lambda_2, \cdots, \lambda_n) P^{-1} \in \mathcal{M}_{\text{diag}^n}(S)$, we define the evaluation of $f$ at $M$ as $f(M) = P \,\,\diag(f(\lambda_1), f(\lambda_2), \cdots, f(\lambda_n)) \,P^{-1}$. This defines the extended function $f:\mathcal{M}_{\text{diag}^n}(S) \rightarrow \mathcal{M}_{\text{diag}^n}$. We denote both the original scalar function as well as the extended function by $f$, which can be disambiguated easily by the input to the function.
\end{definition}

\noindent
For most algebraic/analytic functions it can be shown that one can evaluate $f(M)$ just by substituting $x$ with $M$ in the given algebraic expression of $f$.
\begin{example}~

\begin{itemize}
    \item[i.] If $f(x) = x^p$ is a monomial, then $f(M) = M^p$. %This is true even for non-integer $p$, in which case $M^p = P \,\diag(\lambda_1^p, \lambda_2^p, \cdots, \lambda_n^p) P^{-1}$.
    This is because $M^p = (P \Lambda P^{-1}) (P \Lambda P^{-1}) \cdots (P \Lambda P^{-1}) = P \Lambda^p P^{-1} = P \,\diag_l(\lambda_l^p) \, P^{-1} = f(M)$ (here $\Lambda = \diag(\lambda_1, \lambda_2, \cdots, \lambda_n)$, and $\diag_l(a_l)$ is a compact notation for $\diag(a_1, a_2, \cdots, a_n)$).
    \item[ii.] If $f(x) = \sum_{k=0}^m a_k x^k$ is a polynomial, $f(M) = \sum_{k=0}^m a_k M^k$. 
    \item[iii.] If $f(x) = \frac{p(x)}{q(x)}$ is a rational function, where $p$ and $q$ are polynomials, then $f(M) = p(M) \left( q(M) \right)^{-1}$. The domain of the extended function is the set of all diagonalizable matrices with eigenvalues which are not roots of $q$.
    \item[iv.] If $f(x)$ is continuous in $[a,b]\subseteq \mathbb{R}$, one can construct a sequence of polynomials, $f_1, f_2, \cdots$ that converges uniformly to $f$ in $[a,b]$ (Stone-Weierstrass approximation theorem). Then $f(M)$ can be evaluated as $\lim_{k\rightarrow \infty} f_k(M)$. %of the sequence $\{f_k(M)\}_{k=1,2,\cdots}$.
\end{itemize}
\end{example}
\noindent
%Moving forward, we will restrict ourselves to real functions, $f:\mathbb{R}\times\mathbb{R}\rightarrow \mathbb{R}$, and matrices with real eigenvalues.

% \noindent
% Consider the matrix $M$ parametrized by a real parameter $t$. 
% It is easy to verify that for a monomial, $f(x) = x^p$, the derivative of $f(M(t)) = (M(t))^p$ with respect to $t$ is $\frac{f(M(t))}{dt} = \sum_{k=0}^{p-1} (M(t))^j \frac{d M(t)}{dt} (M(t))^{p-1-j}$. The result of the following lemma is a simple consequence of this observation for polynomial $f$.
% % The following lemma is straight-forward to prove for monomials and polynomials
% However, the extension of this result to $\mathcal{C}^1$ functions require certain technical machinaries that we borrow from~\citep{bhatia} in order to present the result in its full generality.

%\begin{definition}
%	For $S\subseteq \mathbb{K}$,
%	define $\mathcal{D}^n_\text{diag}(S) \subseteq \mathbb{K}^{n\times n}$ to be the set of $n\times n$ diagonalizable matrices whose eigenvalues lie in $S$.
%\end{definition}

\subsection{Trace of $f(M)$}

One scalar function of particular interest is the trace for $f(M)$.
From Definition~\ref{def:function-single-var} we note that $(\tr\circ f)(M) := \tr(f(M)) = \sum_{j=1}^n f(\lambda_j)$.

The following result is easy to show for a monomial, and hence a polynomial, $f$, and thus can be generalized to $\mathcal{C}^1$ functions on a compact subset of a real line using continuity arguments and Stone-Weierstrass approximation theorem.

\begin{proposition}[Derivative of Trace of a Differentiable function of $M$~\citep{carlen}] \label{lemma:derivative}
Consider real-eigenvalued diagonaliable matrix, $M\in \mathcal{M}_{\text{diag}^n}$, parametrized by a real parameter $t$ such that the elements of $M(t)$ are $\mathcal{C}^1$ in $t$.
Let $f:[a,b]\rightarrow\mathbb{R}$ be a function that is $\mathcal{C}^1$ in $[a,b]$. Then,
\begin{equation} \label{eq:trace-fun}
    \frac{d}{dt} (\tr\circ f)(M(t)) ~=~ \tr\left( f'(M(t)) \frac{dM(t)}{dt} \right) 
\end{equation}
for all $t$ such that $M(t)\in\mathcal{M}_{\text{diag}^n}([a,b])$, and $\left\|\frac{d M(t)}{dt}\right\|, \|(P(t))^{-1}\|$ are bounded (where $P(t)$ is the normalized modal matrix made up of the unit eigenvectors of $M(t)$ as its columns).
\end{proposition}

\begin{corollary}[Monotonic function~\citep{carlen}] \label{cor:monotonic}
Let $f:[a,b]\rightarrow\mathbb{R}$ be a function that is $\mathcal{C}^1$ and monotonically increasing (resp. decreasing) in $[a,b]$.
Consider real-eigenvalued diagonaliable matrix, $M\in \mathcal{M}_{\text{diag}^n}$, parametrized by a real parameter $t$.
Let 

\noindent
\[ \mathcal{T} = \left\{ t ~\Big|~ M(t)\in\mathcal{M}_{\text{diag}^n}([a,b]), ~\frac{d M(t)}{dt} \cgeq 0, \text{and $\left\|\frac{d M(t)}{dt}\right\|, \|(P(t))^{-1}\|$ are bounded} \right\} \]

\noindent
%Let $\mathcal{T}$ be the set of $t$
% such that the elements of $M(t)$ are $\mathcal{C}^1$ in $t$, $\frac{d M(t)}{dt}$ is positive semi-definite for all $t$ such that $M(t)\in\mathcal{M}_{\text{diag}^n}([a,b])$, and $\left\|\frac{d M(t)}{dt}\right\|, \|(P(t))^{-1}\|$ are bounded.
Then, $g(t) = \tr\left( f\left(M(t) \right) \right)$ is monotonically increasing (resp. decreasing) in $\mathcal{T}$.
\end{corollary}

\begin{example}
	Consider $f(x) = x^2$ which is monotonically increasing in $\mathbb{R}_{\geq 0}$. Let $M(t) = M_0 + M_1 t$, where both $M_0$ and $M_1$ are symmetric positive semi-definite. Then the function $g(t) = \tr ((M_0 + M_1 t)^2)$ is monotonically increasing in $\mathbb{R}_{\geq 0}$.
\end{example}

\subsubsection{Hermitian Matrices}

We now specialize to $\mathcal{M}_{\text{norm}^n} \subset \mathcal{M}_{\text{diag}^n}$.
It is well-known that eigenvalues of convex combination of two normal matrices lie in the convex hull of the eigenvalues of the two matrices~\citep{bhatia}.
However, in general, convex combination of two normal matrices is not normal. So we further restrict ourselves to the set of Herminitian matrices, $\mathcal{M}_{\text{herm}^n} \subset \mathcal{M}_{\text{norm}^n}$, which is closed under linear combination. Most of these results can also be extended to skew-Hermitian matrices which are also closed under linear combination.

\begin{lemma} \label{lemma:convex-set}
	If $S\subseteq\mathbb{R}$ is a convex set ($S$ can be an open, a closed or a clopen interval), so is $\mathcal{M}_{\text{herm}^n}(S)$.
\end{lemma}

%\begin{proof}
%	Follows from the fact that the eigenvalues of convex combination of two matrices line in the convex hull of the eigenvalues of the two matrices~\citep{bhatia}.
%	
%	Let $A,B\in \mathcal{D}^n_{\mathcal{C}}$. For $\alpha\in [0,1]$, let $\lambda$ be an eigenvalue of $\alpha A + (1-\alpha)B$ and $\mathbf{u}$ the corresponding normalized eigenvector.
%	Thus,
%	%	\begin{eqnarray*}
%	$ \lambda ~=~ \mathbf{u}^* \left( \alpha A + (1-\alpha)B \right) \mathbf{u}
%	~=~ \alpha \mathbf{u}^* A \mathbf{u}  + (1-\alpha) \mathbf{u}^* B \mathbf{u}$
%	%	\end{eqnarray*}
%\end{proof}

\begin{definition}[Operator Monotone and Operator Convex Functions~\citep{bhatia}]
	Suppose $S\subseteq \mathbb{R}$.
	\begin{itemize}
		\item A function $f:S\rightarrow\mathbb{R}$ is called operator monotone if $A \cgeq B ~\Rightarrow~ f(A) \cgeq f(B)$ for all $A,B\in \mathcal{M}_{\text{herm}^n}(S)$, for all $n\in\mathbb{Z}_{+}$.
		\item A function $f:S\rightarrow\mathbb{R}$ is called operator convex in if 
		$ f\left( (1-\alpha)A + \alpha B \right) ~\cgeq~ (1-\alpha) f(A)+\alpha f(B) $
		for all $A,B\in \mathcal{M}_{\text{herm}^n}(S)$, for all $n\in\mathbb{Z}_{+}$.
	\end{itemize}
\end{definition}

It is easy to check that if $f$ is operator convex (resp. operator monotone), then $\tr\circ f$ is convex (resp. monotonically increasing).

However, in general, if the scalar function $f$ is a monotonic (or convex) function, then it is not necessarily an operator monotone (or operator convex) function.
Examples of this can be found in~\citep{bhatia,carlen}.
However, $\tr\circ f$ does inherit the monotonicity and/or convexity of $f$.

In the rest of the paper, whenever we state \emph{``$f$ is monotonic/convex''}, we mean that the scalar function $f$ is monotonic/convex.

\begin{proposition}[Monotonicity and Convexity Properties of Trace of a Function~\citep{carlen}]
	Let $f:\mathbb{R}\rightarrow\mathbb{R}$ be a function that is $\mathcal{C}^0$ on a convex set $S\subseteq\mathbb{R}$.
	%Define $\tr \circ f: \mathcal{M}_{\text{herm}^n}(S) \rightarrow \mathbb{R}$ as $(\tr\circ f)(M)= \tr \left( f(M) \right)$.
%	Then,
%	\begin{itemize}
%		\item[i.] If $f$ is monotonically increasing (resp. decreasing) in $S$, then $\tr \circ f$ is monotonically increasing (resp. decreasing) in $\mathcal{M}_{\text{herm}^n}(S)$ (\emph{i.e.}, if $A\cgeq B$ then $\tr ( f(A)) \geq \tr ( f(B))$).
%		\item[ii.] 
		If $f$ is convex (resp. concave) in $S$, then $\tr \circ f$ is convex (resp. concave) in $\mathcal{M}_{\text{herm}^n}(S)$.
%	\end{itemize}
\end{proposition}

%\begin{proof}
%	\emph{Proof of i.}:
%	Suppose $A\cgeq B$. We need to prove $f_{\tr}(A) \geq f_{\tr}(A)$.
%	By definition of the partial order $\cgeq$, for any unit-vector 
%\end{proof}

\begin{corollary}
	Suppose $f:\mathbb{R}\rightarrow \mathbb{R}$ is a $\mathcal{C}^0$ function on a convex set $S\subseteq \mathbb{R}$.
	Let $M$ be a Hermitian matrix parametrized linearly by $t$ so that we can write $M(t) = M_0 + M_1 t$, where both $M_0$ and $M_1$ are Hermitian and $M_1$ is positive semi-definite.
	Define $\mathcal{T} = \{t ~|~ M(t) \in \mathcal{M}_{\text{herm}^n}(S) \} \subseteq \mathbb{R}$ (it's easy to check using Lemma~\ref{lemma:convex-set} that $\mathcal{T}$ is a convex set).
	Define $g = \tr\circ f \circ M: \mathcal{T} \rightarrow \mathbb{R}$ as $g(t) = \tr(f(M(t)))$.
%	Then,
%	\begin{itemize}
%		\item[i.] If $f$ is monotonically increasing (resp. decreasing) in $S$, then $g$ is monotonically increasing (resp. decreasing) in $\mathcal{T}$.
%		\item[i.] 
		If $f$ is convex (resp. concave) in $S$, then $g$ is convex (resp. concave) in $\mathcal{T}$.
%	\end{itemize}
\end{corollary}

\section{Multi-variable Functions}

\subsection{Introduction: Function of Two Variables}

We start with generalization of the definition of evaluation of a function at matrices to the case of function of two variables.

\begin{definition}[Extension of a Scalar Function of Two Variables to Diagonalizable Matrices]\label{def:function-double-var}
Consider a function $f:S_1 \times S_2 \rightarrow \mathbb{C}$
%(where $S_1, S_2\subseteq \mathbb{K}$, where $\mathbb{K}$ is either $\mathbb{R}$ or $\mathbb{C}$).
(where $S_j$ are subsets of $\mathbb{C}$, with the possibility that $S_j \subseteq \mathbb{R}$ and/or $f$ is a real-valued function).
Given two $n\times n$ diagonalizable matrices, $M = P \,\diag(\mu_1, \mu_2, \cdots, \mu_n) P^{-1} \in \mathcal{M}_{\text{diag}^n}(S_1)$ and $N = Q \,\diag(\nu_1, \nu_2, \cdots, \nu_n) Q^{-1} \in \mathcal{M}_{\text{diag}^n}(S_2)$, we define the evaluation of $f$ at $(M,N)$ as 
$f:\mathcal{M}_{\text{diag}^n}(S_1) \times \mathcal{M}_{\text{diag}^n}(S_2) \rightarrow \mathbb{C}^{n\times n} \otimes \mathbb{C}^{n\times n} \simeq \mathbb{C}^{n^2 \times n^2}$ given by
\begin{eqnarray*}
f(M,N) ~~=~~ \left(P\otimes Q \right) ~
     \diag_{k,l}\!\left(f(\mu_k,\nu_l)\right)
    \left(P^{-1} \otimes Q^{-1} \right) 
\end{eqnarray*}
where `$\otimes$' is the tensor/Kroneker product,
and,
\begin{eqnarray*}
	\diag_{k,l}\!\left(f(\mu_k,\nu_l)\right) & = & \diag\Big(
	f(\mu_1,\nu_1), f(\mu_1,\nu_2), \cdots, f(\mu_1,\nu_n), \\
	& & ~~\qquad f(\mu_2,\nu_1), \cdots, f(\mu_2,\nu_n), \\
	& & \qquad\qquad \cdots, \cdots, f(\mu_n,\nu_1), \cdots, f(\mu_n,\nu_n)
	\Big)
\end{eqnarray*}
%For compactness, from now on we will write the above-mentioned diagonal matrix as ~$\diag_{k,l}\left(f(\mu_k,\nu_l)\right)$.
If the $(i,j)$-th element of a matrix $M$ is denoted $[M]_i^j$, in index notation the above definition is
\[ [ f(N,M) ]_{i_1 i_2}^{j_1 j_2} ~=~ \sum_{k_1,k_2} [P]_{i_1}^{k_1} [Q]_{i_2}^{k_2} \, f(\mu_{k_1}, \nu_{k_2}) \, [P^{-1}]_{k_1}^{j_1} [Q^{-1}]_{k_2}^{j_2} \]
\end{definition}

\begin{example} ~
    \begin{itemize}
        \item[i.] For a monomial, $f(x,y) = x^p y^q$, its matrix evaluation is $f(M,N) = M^p \otimes N^q$. This is due to the fact that $\diag_{k,l}\left(\mu_k^p \nu_l^q\right) = \diag_{k}\left(\mu_k^p\right) \otimes  \diag_{l}\left(\nu_l^q\right)$, and %the fact that $(A\otimes B)(C\otimes D) = (AC) \otimes (BD)$.
        thus $(P\otimes Q) \left(\diag_{k}\left(\mu_k^p\right) \otimes  \diag_{l}\left(\nu_l^q\right)\right) (P^{-1} \otimes Q^{-1}) = \left( P \diag_{k}\left(\mu_k^p\right) P^{-1} \right) \otimes \left( Q \diag_{k}\left(\nu_k^p\right) Q^{-1} \right)$.
        \item[ii.] For a polynomial, $f(x,y) = \sum_{p,q} a_{pq} x^p y^q$, its matrix evaluation is $f(M,N) = \sum_{pq} a_{pq} M^p \otimes N^q$.
        \item[iii.] For rational function $f(x,y) = \frac{g(x,y)}{h(x,y)}$, the matrix extension is $f(M,N) = g(M,N) \left( h(M,N)\right)^{-1}$.
        \item[iv.] From Definition~\ref{def:function-double-var} we note that $\tr(f(M,N)) = \sum_{j=1}^n \sum_{k=1}^n f(\mu_j,\nu_k)$.
    \end{itemize}
\end{example}

% --------------------------

\subsection{General Case: Multi-Variable Functions}

%\begin{definition}
%	Given a set $W\subseteq \mathbb{C}^n$, we define $\mathcal{M}^m_{\text{typ}^n}(W)$ to be the ordered set of $m$ counts of $n\times n$ matrices of type `$\text{typ}$' (where `$\text{typ}$' can be `$\text{diag}$', `$\text{norm}$' or `$\text{herm}$'), $(M_1, M_2, \cdots, M_m)$, such that any ordered choice of one eigenvalue from each of the matrices, $(\lambda_{1 j_1}, \lambda_{2 j_2}, \cdots, \lambda_{m j_m})$ (where $\lambda_{k j_k}$, $k\in\{1,2,\cdots,m\}$, $j_k \in \{1,2,\cdots,n\}$ is an eigenvalue of the matrix $M_k$), lies in $W$.
%	More compactly,
%	{\small \[ \mathcal{M}^m_{\text{typ}^n}(W) = \left\{ (M_1, M_2, \cdots, M_m) \in \mathcal{M}_{\text{typ}^n} \times \mathcal{M}_{\text{typ}^n} \times \cdots \mathcal{M}_{\text{typ}^n} ~|~ (\lambda_{1 j_1}, \lambda_{2 j_2}, \cdots, \lambda_{m j_m}) \in W, ~\lambda_{k j_k} \in \Lambda(M_k) \right\} \] }
%	where $\Lambda(M)$ refers to the set of eigenvalues of the matrix $M$.
%\end{definition}
%
%\begin{example}
%	If $W = (a,b)^m \subseteq \mathbb{R}^m \subset\mathbb{C}^m$ (Cartesian product of $m$ copies of the real interval $(a,b)$), then $\mathcal{M}^m_{\text{typ}^n}(W) = \mathcal{M}_{\text{typ}^n}((a,b)) \times \mathcal{M}_{\text{typ}^n}((a,b)) \times \cdots \times \mathcal{M}_{\text{typ}^n}((a,b))$.
%\end{example}

At this point we generalize the above definition to a function of $m$ variables.
\begin{definition}[Extension of a Scalar Function of $m$ Variables to Diagonalizable Matrices]\label{def:function-mult-var}
%	For $S_j\subseteq \mathbb{C}, j=$ 
	Consider a function $f:S_1\times S_2\times\cdots \times S_m \rightarrow \mathbb{C}$.
	(where $S_j\subseteq \mathbb{C}, j=1,2,\cdots,m$, %where $\mathbb{K}$ is either $\mathbb{R}$ or $\mathbb{C}$).
	with the possibility that $S_j \subseteq \mathbb{R}$ and/or $f$ is real-valued).
	Given $n\times n$ diagonalizable matrices, $M_j = P_j \,\diag(\lambda_{j1}, \lambda_{j2}, \cdots, \lambda_{jn}) P_j^{-1}, ~j=1,2,\cdots,m$, we define the evaluation of $f$ at $(M_1,M_2,\cdots,M_m)$ as $f:\mathcal{M}_{\text{diag}^n}(S_1) \times \mathcal{M}_{\text{diag}^n}(S_2) \times \cdots \times \mathcal{M}_{\text{diag}^n}(S_m) \rightarrow \mathbb{C}^{n\times n} \otimes \mathbb{C}^{n\times n} \otimes \cdots \otimes \mathbb{C}^{n\times n} \simeq \mathbb{C}^{n^m \times n^m}$ given by
	\begin{eqnarray} \label{eq:multi-var-def}
		f(M_1, M_2, \cdots, M_m) ~~=~~ \left( \bigotimes_{j=1}^m P_j \right) ~
		\left( \diag_{j_1,j_2,\cdots,j_m}\!\left(f(\lambda_{1 j_1},\lambda_{2 j_2},\cdots,\lambda_{m j_m})\right) \right)
		\left( \bigotimes_{j=1}^m P_j^{-1} \right)
	\end{eqnarray}
%	where `$\otimes$' is the tensor/Kroneker product,
%	and,
%	\begin{eqnarray*}
%		\diag_{k,l}\!\left(f(\lambda_k,\mu_l)\right) & = & \diag\Big(
%		f(\lambda_1,\mu_1), f(\lambda_1,\mu_2), \cdots, f(\lambda_1,\mu_n), \\
%		& & ~~\qquad f(\lambda_2,\mu_1), \cdots, f(\lambda_2,\mu_n), \\
%		& & \qquad\qquad \cdots, \cdots, f(\lambda_n,\mu_1), \cdots, f(\lambda_n,\mu_n)
%		\Big)
%	\end{eqnarray*}
	%For compactness, from now on we will write the above-mentioned diagonal matrix as ~$\diag_{k,l}\left(f(\lambda_k,\mu_l)\right)$.
%	If the $(i,j)$-th element of a matrix $M$ is denoted $[M]_i^j$, 
	Moving forward, a list of inputs to a function such as $(a_1, a_2, \cdots, a_m)$ will be written as $(a_l)_{l=1}^m$ for compactness.
  	In index notation equation~\eqref{eq:multi-var-def} is
	\[ [ f((M_l)_{l=1}^m) ]_{i_1 i_2 \cdots i_m}^{j_1 j_2 \cdots j_m} ~=~ \sum_{k_1,k_2,\cdots,k_m=1}^n 
	\left( \prod_{l=1}^m [P_l]_{i_l}^{k_l} \right)
	f\left((\lambda_{l k_l})_{l=1}^m\right) 
	\left( \prod_{l=1}^m [P_l^{-1}]_{k_l}^{j_l} \right) \]
\end{definition}
%\noindent
Here $\sum_{k_1,k_2,\cdots,k_m=1}^n$ is used as a compact notation for the $m$ nested summations, $\sum_{k_1=1}^n \sum_{k_2=1}^n \cdots \sum_{k_m=1}^n$.
Moving forward we will also use the compact notation $\bigtimes_{j=1}^m W_j = W_1 \times W_2 \times \cdots \times W_m$ to denote the Cartesian product of $m$ sets.

\subsection{Trace of $f\left((M_l)_{l=1}^m\right)$}

We first note that $\tr \circ f\left((M_l)_{l=1}^m\right) = \tr\left( f\left((M_l)_{l=1}^m\right) \right) =
% \displaystyle \sum_{j_1,j_2,\cdots,j_m=1}^n f(\lambda_{1 j_1},\lambda_{2 j_2},\cdots,\lambda_{m j_m}) =
  \sum_{j_1,\cdots,j_m=1}^n f\left((\lambda_{l j_l})_{l=1}^m\right)$.
%\noindent
%Moving forward we will restrict ourselves to real functions, $f:\mathbb{R}\times\mathbb{R}\rightarrow \mathbb{R}$, and matrices with real eigenvalues.

\begin{proposition}[Derivative of Trace of a Differentiable function of $(M_l)_{l=1}^m$] \label{prop:derivative-multi-var}
Let $S_j\subseteq \mathbb{R}, j=1,2,\cdots,m$ be locally-compact subsets of the real line. %(open, closed or clopen intervals).
Let $f:\bigtimes_{j=1}^m S_j \rightarrow \mathbb{R}$ be a real-valued function of class $\mathcal{C}^m$. % in $\bigtimes_{j=1}^m S_j\subseteq\mathbb{R}^m$.
Consider real eigen-valued diagonaliable matrices, $M_l \in \mathcal{M}_{\text{diag}^n}, ~l=1,2,\cdots,m$, parametrized by a real parameter $t$ such that the elements of $M_l(t)$ are $\mathcal{C}^1$ in $t$.
 Then,
\begin{equation} \label{eq:trace-fun}
    \frac{d}{dt} \tr(f(M_l(t))_{l=1}^n) 
    ~=~ \tr\left( \sum_{k=1}^m \partial_k f\left((M_l(t))_{l=1}^m\right) \left(I^{\otimes(k-1)} \otimes \frac{dM_k(t)}{dt} \otimes I^{\otimes(m-k)} \right)
            \right) 
\end{equation}
for all $t$ such that 
$M_j(t) \in \mathcal{M}_{\text{diag}^n} (S_j)$,
%any eigenvalue tuple of the form $\left(\lambda_{1 j_1},\lambda_{2 j_2},\cdots,\lambda_{m j_m}\right)$ (where $\lambda_{l j_l}$ is an eigenvalue of $M_l(t), l=1,2,\cdots,m$) is in $\mathcal{E}$, 
and $\left\|\frac{d M_l(t)}{dt}\right\|$ \& $\|(P_l(t))^{-1}\|$ are bounded for all $j\in\{1,2,\cdots,m\}$ (where $P_l(t)$ are the normalized modal matrices made up of unit eigenvectors of $M_l(t)$ as their columns).

Here $\partial_k f$ refers to the partial derivative of the function $f$ with respect to its $k$-th input, and $I^{\otimes r}$ refers to $r$ times tensor product of the identity matrix with itself.
\end{proposition}

\begin{proof}
%    \todo{Prove for monomial, then polynomial and then $\mathcal{C}^1$ functions.}
%    
    We will start by proving the statement for a monomial $f$, then we will prove it for a polynomial, and finally for general $\mathcal{C}^m$ functions.
    
    \vspace{0.5em}\noindent
    \emph{Monomial:} If $f((x)_{l=1}^m) = \prod_{l=1}^m {x_l}^{p_l}$, we have\\
    % \begin{eqnarray*}
    $
    \frac{d}{dt} \tr\left(f\left( (M_l(t))_{l=1}^m \right)\right) 
    ~~=~~ \frac{d}{dt} \tr\left( \bigotimes_{l=1}^m (M_l(t))^{p_l} \right) $
    
    $=~~ \frac{d}{dt} \prod_{l=1}^m \tr\left( (M_l(t))^{p_l} \right)
%    ~~=~~ \tr\left( \frac{d}{dt} \bigotimes_{l=1}^m (M_l(t))^{p_l} \right) 
    $ \qquad\qquad{\small(since $\tr(A\otimes B) = \tr(A) \tr(B)$)}
    
%    $=~~ \tr\left( \sum_{k=1}^m  \left( \bigotimes_{l=1}^{k-1} (M_l(t))^{p_l} \right) \otimes
%     \left( \frac{d}{dt} (M_k(t))^{p_k} \right)
%     \otimes \left( \bigotimes_{l=k+1}^{m} (M_l(t))^{p_l} \right) \right) $ \qquad(product rule)
%    
%     $=~~ \tr\left( \sum_{k=1}^m  \left( \bigotimes_{l=1}^{k-1} (M_l(t))^{p_l} \right) \otimes
%     	\left( \sum_{j=0}^{p-1} (M(t))^j \frac{d M(t)}{dt} (M(t))^{p-1-j} \right)
%     	\otimes \left( \bigotimes_{l=k+1}^{m} (M_l(t))^{p_l} \right) \right) %\\
%     $
    $= \sum_{k=1}^m \left( \prod_{l=1}^{k-1} \tr \left( (M_l(t))^{p_l} \right) \right) \tr\left(\frac{d}{dt} (M_k(t))^{p_k} \right) \left( \prod_{l=k+1}^{m} \tr \left( (M_l(t))^{p_l} \right) \right)$

    $= \sum_{k=1}^m p_k \left( \prod_{l=1}^{k-1} \tr \left( (M_l(t))^{p_l} \right) \right) \tr\left( (M_k(t))^{p_k-1} \frac{dM_k(t)}{dt} \right) \left( \prod_{l=k+1}^{m} \tr \left( (M_l(t))^{p_l} \right) \right)$
    
    	\qquad\qquad {\small (since $\tr\left(\frac{d}{dt} (M_k(t))^{p_k} \right) = \sum_{j=0}^{p_k-1}  \tr\left(  (M_k(t))^j \frac{d M_k(t)}{dt} (M_k(t))^{p_k-1-j} \right) %= \sum_{j=0}^{p_k-1}  \tr\left(  (M_k(t))^{p_k-1} \frac{d M_k(t)}{dt}  \right) 
    	= p_k\tr\left( M_k(t)^{p_k-1} \frac{dM_k(t)}{dt} \right) $)} %\\
    
    $= \sum_{k=1}^m p_k \tr \left( \bigotimes_{l=1}^{k-1} (M_l(t))^{p_l}  ~\otimes~ \left( (M_k(t))^{p_k-1} \frac{dM_k(t)}{dt} \right) ~\otimes~ \bigotimes_{l=k+1}^{m} \tr \left( (M_l(t))^{p_l} \right) \right)$
    
    $= \sum_{k=1}^m \tr \left( \left( p_k \bigotimes_{l=1}^{k-1} (M_l(t))^{p_l}  ~\otimes~  (M_k(t))^{p_k-1}  ~\otimes~ \bigotimes_{l=k+1}^{m} (M_l(t))^{p_l} \right) \left( I^{\otimes(k-1)} \otimes \frac{dM_k(t)}{dt} \otimes I^{\otimes(m-k)} \right) \right)$
    
    	\qquad\qquad\qquad {\small (since $A \otimes (BC) \otimes D = (A\otimes B\otimes D)(I \otimes C \otimes I)$.)}
    
    $= \tr\left( \sum_{k=1}^m \partial_k f\left((M_l(t))_{l=1}^m\right) \left(I^{\otimes(k-1)} \otimes \frac{dM_k(t)}{dt} \otimes I^{\otimes(m-k)} \right) \right)$
    
%    $=~~ \sum_{j=0}^{p-1}  \tr\left(  (M(t))^{p-1} \frac{d M(t)}{dt}  \right) %\\
%    ~~=~~ p\tr\left( M(t)^{p-1} \frac{dM(t)}{dt} \right)
%    $
    
%    \qquad $~~=~~ \tr\left(p M(t)^{p-1} \frac{dM(t)}{dt} \right) 
%    ~~=~~ \tr\left( f'(M(t)) \frac{dM(t)}{dt} \right)$
    % \end{eqnarray*}
    
    \vspace{0.5em}\noindent
    \emph{Polynomial:} 
%    If $f(x) = \sum_{k=0}^m a_k x^k$, we have\\
%    $
%    \frac{d}{dt} \tr(f(M(t))) 
%    ~~=~~ \frac{d}{dt} \tr(\sum_{k=0}^m a_k (M(t))^k) 
%    ~~=~~ \sum_{k=0}^m a_k \frac{d}{dt} \tr\left( (M(t))^k \right)
%    ~~=~~ \sum_{k=0}^m a_k \tr\left( k (M(t))^{k-1} \frac{dM(t)}{dt} \right) 
%    $
%    
%    $
%    ~~=~~ \tr\left( \sum_{k=0}^m a_k  k (M(t))^{k-1} \frac{dM(t)}{dt} \right)
%    ~~=~~ \tr\left( f'(M(t)) \frac{dM(t)}{dt} \right)
%    $
	Due to linearity of trace and derivative, the result naturally extends to polynomials.
    
    \vspace{0.5em}\noindent
    \emph{$\mathcal{C}^m$ function:}
    % Fist we note that in general the eigenvalues of matrices vary smoothly with the matrix.
    We first note that using the Implicit Function Theorem on the characteristic polynomial, $Q(\lambda; t)$, of $M_l(t)$, it can be shown that the eigenvalues of the matrix $M_l(t)$ (for a $t$ in an open subset, $S_l$, of the real line, in a neighborhood of which all solutions to the characteristic polynomial are real), vary smoothly with $t$~\citep{rahman2002analytic}. Hence the derivatives of the eigenvalues of $M_l(t)$, $\left(\lambda'_{lj}(t)\right)_{j=1,2,\cdots,n}$ are defined. %, which are real for all $t\in (a,b)$
    By Weyly's Perturbation Theorem~\citep{bhatia}, $\left| \lambda_{lj}(t+\Delta t) - \lambda_{lj}(t) \right| \leq \| M_l(t+\Delta t) - M_l(t) \|_2$ (where, $\lambda_{lj}(t)$ denotes the $j$-th eigenvalue of $M_l(t)$ in decreasing order of value, and $\|\cdot\|_2$ denotes the induced $2$-norm). With $\Delta t \rightarrow 0$, this gives $\left| \lambda'_{lj}(t) \right| \leq \left\|\frac{dM_l(t)}{dt}\right\|_2$.
    % Thus $\lim_{\Delta t \rightarrow 0} \frac{\left| \lambda_j(t+\Delta t) - \lambda_jM(t) \right|}{\Delta t}$ is bounded by $\left\|\frac{dM(t)}{dt}\right\|_2$ (even if the left and right sided limits do not match).
    
    Given a function $f$ that is $\mathcal{C}^m$ in $\bigtimes_{l=1}^m S_l\subseteq \mathbb{R}^m$, one can construct a sequence of polynomials, $f_1, f_2, \cdots$ that converges uniformly to $f$ as well as $\partial_k f_1, \partial_k f_2, \cdots$ converges uniformly to $\partial_k f$ in $\bigtimes_{l=1}^m S_l$ for all $k=1,2,\cdots,m$ (If needed, construct an approximation, $\widetilde{f}$, that is $\mathcal{C}^m$ and agrees with $f$ in a neighborhood, $\bigtimes_{l=1}^m (a_l, b_l)$, of $\bigtimes_{l=1}^m S_j$, but vanishes outside it. Use Stone-Weierstrass approximation theorem to construct a sequence of polynomial approximations, $\{p_k\}_{k=1,2,\cdots}$, of the $m$-th mixed derivative, $\partial_1 \partial_2 \cdots \partial_m \widetilde{f}$, and then integrate it $m$ times, $f_k((x_l)_{l=1}^m) = \int_{a_1}^{x_1} \int_{a_2}^{x_2} \cdots \int_{a_m}^{x_m} p_k((s_l)_{l=1}^m) ds_m ds_{m-1} \cdots ds_1$, to get the necessary sequence of polynomial approximations, $\{f_k\}_{k=1,2,\cdots}$, of $\widetilde{f}$, and hence of $f$).
    If $\epsilon$ is any positive real number, there exists a $N_0$ such that 
    % $\sup_{x\in[a,b]} |f_N(x) - f(x)| < \epsilon$ and 
    $\sup_{x_l\in S_l} |\partial_k f_N((x)_{l=1}^n) - \partial_k f((x)_{l=1}^n)| < \epsilon$ for all $N\geq N_0$, $k=1,2,\cdots,m$~\citep{abbott2001understanding}.
    
    Suppose $M_l(t) = P_l(t) \,\diag(\lambda_{l1}(t),\lambda_{l2}(t),\cdots, \lambda_{ln}(t)) (P_l(t))^{-1}$ is the diagonalization of the matrix.
    %
    % Then,
    % \begin{eqnarray*}
    %     \tr\left( \left( \frac{d M(t)}{dt} \right)^2 \right) & = &
    % \end{eqnarray*}
    %
    % Since the (unordered) set of eigenvalues change smoothly with the matrix elements, 
    % Thus,
    % $\tr\left(\frac{d \, M(t) }{dt}\right) = 
    % \frac{d}{dt} \tr\left( M(t) \right) = \frac{d}{dt} \sum_{l=1}^n \lambda_l(t) = \sum_{l=1}^n \lambda'_l(t) \leq \sum_{l=1}^n |\lambda'_l(t) |
    % % \tr\left( \frac{d P(t)}{dt} \,\diag_l(\lambda_l(t)) (P(t))^{-1} + P(t) \,\diag_l(\lambda'_l(t)) (P(t))^{-1} + P(t) \,\diag_l(\lambda_l(t)) \frac{d (P(t))^{-1}}{dt} \right)
    % $
    Thus, for a $t$ such that $\lambda_{lk}(t) \in S_l, ~\forall l,k$,
    \begin{eqnarray*}
        & & \left| \frac{d}{dt} \tr(f_N((M_l(t))_{l=1}^m)) -  \frac{d}{dt} \tr(f((M(t))_{l=1}^m)) \right| \\
        & = & \left| \frac{d}{dt} \sum_{j_1,j_2,\cdots=1}^n f_N\left( (\lambda_{l j_l}(t))_{l=1}^m \right) ~~-~~  \frac{d}{dt} \sum_{j_1,j_2,\cdots=1}^n f\left( (\lambda_{l j_l}(t))_{l=1}^m \right) \right| \\
        % & = & \left| \frac{d}{dt} \sum_{l=1}^n \left( f_N(\lambda_l(t)) - f(\lambda_l(t)) \right) \right| \\
        & = & \left| \sum_{j_1,j_2,\cdots=1}^n \sum_{k=1}^m \left( \partial_k f_N\left( (\lambda_{l j_l}(t))_{l=1}^m \right) - \partial_k f\left( (\lambda_{l j_l}(t))_{l=1}^m \right) \right) \lambda'_{k j_k}(t) \right| \\
        & \leq & \sum_{j_1,j_2,\cdots=1}^n \sum_{k=1}^m \left| \partial_k f_N\left( (\lambda_{l j_l}(t))_{l=1}^m \right) - \partial_k f\left( (\lambda_{l j_l}(t))_{l=1}^m \right) \left| ~\right| \lambda'_{k j_k}(t) \right| %\\
        %& \leq & \epsilon \sum_{l=1}^n \left| \partial_k f_N\left( (\lambda_{l j_l}(t))_{l=1}^m \right) - \partial_k f\left( (\lambda_{l j_l}(t))_{l=1}^m \right) \left| ~\right| \lambda'_{k j_k}(t) \right|
        % ~\leq~ \epsilon \sqrt{n \sum_{l=1}^n \left| \lambda'_l(t) \right|^2} %~=~ \epsilon \left\|\frac{d M(t)}{dt}\right\|_F
        ~~\leq~~ n^m \,m \,\epsilon \left\|\frac{dM(t)}{dt}\right\|_2
    \end{eqnarray*}
    This proves that the left-hand-side of equation~\eqref{eq:trace-fun} gets arbitrarily close to $\frac{d}{dt} \tr\left(f_N\left((M_l(t))_{l=1}^m \right) \right)$ as $N\rightarrow \infty$.
    
    % Before proving the same for the right-hand-side, we note that using Cauchy-Schwarz inequality and sub-multiplicativity of Frobenius norm,
    % \begin{equation}
    %     \tr(D P^{-1} B P) \leq \|D\|_F \|P^{-1} B P \|_F = \|D\|_F \|P^{-1}\|_F \|B\|_F \|P\|_F
    % \end{equation}
    
    Again,
    \begin{eqnarray*}
        & & \left| \tr\left( \sum_{k=1}^m \partial_k f_N\left((M_l(t))_{l=1}^m\right) \left(I^{\otimes(k-1)} \otimes \frac{dM_k(t)}{dt} \otimes I^{\otimes(m-k)} \right) \right) \right. \\ & & \qquad
        	\left. - \tr\left( \sum_{k=1}^m \partial_k f\left((M_l(t))_{l=1}^m\right) \left(I^{\otimes(k-1)} \otimes \frac{dM_k(t)}{dt} \otimes I^{\otimes(m-k)} \right) \right) \right| \\
%        & = & \left| \tr\left( \left( f'_N(M(t)) - f'(M(t)) \right) \frac{dM(t)}{dt} \right) \right| \\
        & = & \left| \sum_{k=1}^m \tr\left( \left(\bigotimes_{j=1}^m P_j(t)\right) ~\diag_{j_1,j_2,\cdots,j_m} \left( \partial_k f_N((\lambda_{l j_l}(t))_{l=1}^m) - \partial_k f((\lambda_{l j_l}(t))_{l=1}^m) \right) \left(\bigotimes_{j=1}^m (P_j(t))^{-1} \right) \right. \right. \\ & &
        	\qquad\qquad\qquad\left. \left. \left(I^{\otimes(k-1)} \otimes \frac{dM_k(t)}{dt} \otimes I^{\otimes(m-k)} \right) \right) \right| \\
        & = & \left| \sum_{k=1}^m \tr\Bigg( \diag_{j_1,j_2,\cdots,j_m} \left( \partial_k f_N((\lambda_{l j_l}(t))_{l=1}^m) - \partial_k f((\lambda_{l j_l}(t))_{l=1}^m) \right) \right. \\ & &
        \qquad\qquad\qquad\left. \left(\bigotimes_{j=1}^m (P_j(t))^{-1} \right) \left(I^{\otimes(k-1)} \otimes \frac{dM_k(t)}{dt} \otimes I^{\otimes(m-k)} \right) \left(\bigotimes_{j=1}^m P_j(t)\right) \Bigg) \right| \\
%        	
%        & = & \left| \tr\left( \diag_l \left( f'_N(\lambda_l(t)) - f'(\lambda_l(t)) \right) ~~(P(t))^{-1} \frac{dM(t)}{dt} P(t) \right) \right| \\
        & \leq & \mathlarger{\sum}_{k=1}^m \left\| \diag_{j_1,j_2,\cdots,j_m} \left( \partial_k f_N((\lambda_{l j_l}(t))_{l=1}^m) - \partial_k f((\lambda_{l j_l}(t))_{l=1}^m) \right) \right\|_F \\ & & 
        	\qquad\qquad \left\|  \left(\bigotimes_{j=1}^m (P_j(t))^{-1} \right) \left(I^{\otimes(k-1)} \otimes \frac{dM_k(t)}{dt} \otimes I^{\otimes(m-k)} \right) \left(\bigotimes_{j=1}^m P_j(t)\right) \right\|_F \\ & & \qquad\qquad\qquad\text{(Using Cauchy-Schwarz inequality $\tr(A^* B) \leq \sqrt{\tr(A^* A) \tr(B^* B)} = \|A\|_F \|B\|_F$.)} \\
        % \\
%        \end{eqnarray*}\begin{eqnarray*}
        % & \leq & \sqrt{ \left| \tr\left( P(t) ~\diag_l \left( f'_N(\lambda_l(t)) - f'(\lambda_l(t)) \right)^2 (P(t))^{-1} \right) \right|
        % \left| \tr\left(\left( \frac{dM(t)}{dt} \right)^2 \right) \right| }\quad\text{(By Cauchy-Schwarz inequality,} \\ & & \qquad\qquad\text{$\tr(AB) \leq \sqrt{\tr(A^T A) \tr(B^T B)}$)} \\
        % & = & \left| \sum_{l=1}^n \left( f'_N(\lambda_l(t)) - f'(\lambda_l(t)) \right) \right|  ~\left| \tr\left( \frac{dM(t)}{dt} \right) \right| \\
        & = & \mathlarger{\sum}_{k=1}^m ~\sqrt{ \sum_{j_1,j_2,\cdots,j_m=1}^n \left| \partial_k f_N((\lambda_{l j_l}(t))_{l=1}^m) - \partial_k f((\lambda_{l j_l}(t))_{l=1}^m) \right|^2 } \\ & &
        	\qquad\qquad\qquad\qquad ~\left\|  I^{\otimes(k-1)} \otimes \left( (P_k(t))^{-1} \frac{dM_k(t)}{dt} P_k(t) \right) \otimes I^{\otimes(m-k)}  \right\|_F \\
        & \leq & \sum_{k=1}^m \sqrt{n^m \epsilon^2} ~ \left(n^{m-1} \left\|(P_k(t))^{-1} \frac{dM_k(t)}{dt} P_k(t)\right\|_F\right) \\ & &  \qquad\qquad\text{\small(since $\|A\otimes B\|_F^2 = \tr((A\otimes B)^* (A\otimes B) ) = \tr((A^* A)\otimes (B^* B)) = \|A\|_F^2 \|B\|_F^2$)}\\
        & \leq & \epsilon ~n^{\frac{3m}{2}-1} ~\sum_{k=1}^m \left\| (P_k(t))^{-1} \right\|_F \left\| \frac{dM_k(t)}{dt} \right\|_F \sqrt{n} \qquad\text{(since the columns of $P_k(t)$ are unit vectors.)}\\
        % & \leq & \epsilon \, n \sum_{l=1}^n |\lambda'_l(t) | 
        % \quad\text{(since $\textstyle \tr\left(\frac{d \, M(t) }{dt}\right) = 
        %             \frac{d}{dt} \tr\left( M(t) \right) = \frac{d}{dt} \sum_{l=1}^n \lambda_l(t) 
        %             %$} \\ & & \qquad\qquad\qquad\qquad\qquad\qquad\qquad\qquad\qquad\qquad\qquad\qquad \text{$
        %             = \sum_{l=1}^n \lambda'_l(t) \leq \sum_{l=1}^n |\lambda'_l(t) |
        %             % \tr\left( \frac{d P(t)}{dt} \,\diag_l(\lambda_l(t)) (P(t))^{-1} + P(t) \,\diag_l(\lambda'_l(t)) (P(t))^{-1} + P(t) \,\diag_l(\lambda_l(t)) \frac{d (P(t))^{-1}}{dt} \right)
        %             $.~)} \\
        & = & \epsilon ~n^{\frac{3m-1}{2}} ~\sum_{k=1}^m \left\| (P_k(t))^{-1} \right\|_F \left\| \frac{dM_k(t)}{dt} \right\|_F 
    \end{eqnarray*}
    This proves that the right-hand-side of equation~\eqref{eq:trace-fun} gets arbitrarily close to $\frac{d}{dt} \tr(f_N(M(t)))$ as $N\rightarrow \infty$.
    
    Hence in the limit as $\epsilon \rightarrow 0$, this proves the result for a $\mathcal{C}^1$ function, $f$.
\end{proof}

\begin{corollary}[Monotonic function] \label{cor:monotonic-multivar}
	Let $f:\bigtimes_{l=1}^m S_l \rightarrow\mathbb{R}$ (where $S_l$ are locally-compact subsets of the real line) be a function that is $\mathcal{C}^m$ and monotonically increasing (resp. decreasing) in $\bigtimes_{l=1}^m S_l$ with respect to each of its inputs.
	Consider real-eigenvalued diagonaliable matrices, $M_l\in \mathcal{M}_{\text{diag}^n}, ~l=1,2,\cdots,m$, parametrized by a real parameter $t$ such that the elements of $M_l(t)$ are $\mathcal{C}^1$ in $t$.
	Let 
	
	\noindent
	\[ \mathcal{T} = \left\{ t ~\Big|~ M_l(t)\in\mathcal{M}_{\text{diag}^n}(S_l), ~\frac{d M_l(t)}{dt} \cgeq 0, \text{and $\left\|\frac{d M_l(t)}{dt}\right\|, \|(P_l(t))^{-1}\|$ are bounded}, ~\forall l=1,2,\cdots,m \right\} \]
	
	\noindent
	%Let $\mathcal{T}$ be the set of $t$
	% such that the elements of $M(t)$ are $\mathcal{C}^1$ in $t$, $\frac{d M(t)}{dt}$ is positive semi-definite for all $t$ such that $M(t)\in\mathcal{M}_{\text{diag}^n}([a,b])$, and $\left\|\frac{d M(t)}{dt}\right\|, \|(P(t))^{-1}\|$ are bounded.
	Then, $g(t) = \tr\left( f\left( (M(t))_{l=1}^m \right) \right)$ is monotonically increasing (resp. decreasing) in $\mathcal{T}$.
\end{corollary}

\begin{proof}
%	\todo{Show positive semi-definite matrix in input to trace.}
	We prove the result for the case when $f$ is monotonically increasing with respect to each of its inputs.
	Since $\frac{d M_l(t)}{dt}$ is positive semi-definite, we can write it as a square, $\frac{d M_l(t)}{dt} = (S_l(t))^2$, where $S_l(t)$ is also a positive semi-definite matrix.
	thus, using Proposition~\ref{prop:derivative-multi-var},
	\begin{eqnarray*}
	& & \frac{d}{dt} \tr(f(M_l(t))_{l=1}^n) \\
	& = & \tr\left( \sum_{k=1}^m \partial_k f\left((M_l(t))_{l=1}^m\right) \left(I^{\otimes(k-1)} \otimes \frac{dM_k(t)}{dt} \otimes I^{\otimes(m-k)} \right) \right) \\
	& = & \tr\left( \sum_{k=1}^m \partial_k f\left((M_l(t))_{l=1}^m\right) \left(I^{\otimes(k-1)} \otimes S_k(t) \otimes I^{\otimes(m-k)} \right) \left(I^{\otimes(k-1)} \otimes S_k(t) \otimes I^{\otimes(m-k)} \right) \right) \\
	& = & \sum_{k=1}^m \tr\left( \left(I^{\otimes(k-1)} \otimes S_k(t) \otimes I^{\otimes(m-k)} \right) \left( \partial_k f\left((M_l(t))_{l=1}^m\right) \right) \left(I^{\otimes(k-1)} \otimes S_k(t) \otimes I^{\otimes(m-k)} \right)  \right) \\
	& = & \sum_{k=1}^m \tr\left( A_k B_k A_k \right)
	\end{eqnarray*}
where $A_k = I^{\otimes(k-1)} \otimes S_k(t) \otimes I^{\otimes(m-k)}$ is positive semi-definite (since $S_k(t)$ is positive semi-definite), and
$B_k = \partial_k f\left((M_l(t))_{l=1}^m\right)$ is also positive semi-definite (since its eigenvalues are $\partial_k f((\lambda_{l j_l})_{l=1}^m), j_l\in\{1,2,\cdots,n\}$, which are non-negative since $f$ is monotonically increasing with respect to each input, and hence $\partial_k f$ is non-negative in $\bigtimes_{k=1}^m S_k$).
Thus $A_k B_k A_k$ is positive semi-definite, hence $\frac{d}{dt} \tr(f(M_l(t))_{l=1}^n)$ is non-negative in $\mathcal{T}$. This proves the result for monotonically increasing case.

A similar proof can be given for the monotonically decreasing case.
\end{proof}

\begin{example}
Consider the function $f(x,y) = x^3 y^5$, which is monotonically increasing in $\mathbb{R}^2$ with respect to each input. Define Hermitian matrices $M_1(t) = C_1 + D_1 t$ and $M_2(t) = C_2 + D_2 t$, where $C_1, D_1, C_2 D_2$ are all Hermitian, and $D_1, D_2$ are positive semi-definite so that $\frac{M_j(t)}{dt} = D_j \cgeq 0$.
Then $g(t) = \tr \left( f(M_1(t), M_2(t)) \right) = \tr\left( (C_1 + D_1 t)^3 \otimes (C_2 + D_2 t)^5 \right)$ is monotonically increasing.
\end{example}	

\subsection{Hermitian Matrices}

We now focus on Hermitian matrices since any linear combination of Hermitian matrices is a Hermitian matrix, and the eigenvalues of convex combination of a collection of Hermitian matrices lie in the convex hull of the eigenvalues of the collection.

\begin{proposition}[Monotonicity Properties of Trace of $f\left((M_l)_{l=1}^m\right)$] \label{prop:monotonic-multivar-hermitian}
	Let $S_j \subset \mathbb{R}, ~j=1,2,\cdots,m$ be convex subsets of the real line.
	Let $f:\bigtimes_{j=1}^m S_j\rightarrow\mathbb{R}$ be a function that is of class $\mathcal{C}^m$. % on $\bigtimes_{j=1}^m S_j \subseteq\mathbb{R}^m$.
	
	If $f$ is monotonically increasing (resp. decreasing) with respect to its inputs in $\bigtimes_{j=1}^m S_j$, then $\tr \circ f$ is monotonically increasing (resp. decreasing) with respect to its inputs in $\bigtimes_{j=1}^m \mathcal{M}_{\text{herm}^n}(S_j)$.
	That is, for $(M_l)_{l=1}^m,$ $(N_l)_{l=1}^m \in \bigtimes_{j=1}^m \mathcal{M}_{\text{herm}^n}(S_j)$ such that $M_l \cgeq N_l, ~\forall\, l=1,2,\cdots,m$, we have $\tr\left(f\left( (M_l)_{l=1}^m \right)\right) \geq \tr\left(f\left( (N_l)_{l=1}^m \right)\right)$ when $f$ is monotonically increasing.
\end{proposition}

\begin{proof}
	Define the parameterization $H_l(t) = N_l + (M_l - N_l) t$ with $t\in[0,1]$ so that $\frac{dH_l(t)}{dt} = M_l - N_l \cgeq 0$.
	This satisfies all the conditions of Corollary~\ref{cor:monotonic-multivar}.
	Thus, $g(t) = \tr\left(f\left( (H_l(t))_{l=1}^m \right)\right)$ is a monotonic function.
	Thus $g(1) \geq g(0)$, that is, $\tr\left(f\left( (M_l)_{l=1}^m \right)\right) \geq \tr\left(f\left( (N_l)_{l=1}^m \right)\right)$.
\end{proof}

\begin{example}
	Consider the function $f(x,y) = x^3 y^5$, which is monotonically increasing in $\mathbb{R}^2$ with respect to each input. If $M_1, N_1, M_2, N_2$ are Hermitian matrices such that $M_1 \cgeq N_1, M_2 \cgeq N_2$, then
	$\tr \left( M_1^3 \otimes M_2^5 \right) \geq \tr \left( N_1^3 \otimes N_2^5 \right)$.
\end{example}

\begin{proposition}[Convexity Properties of Trace of $f\left((M_l)_{l=1}^m\right)$] \label{prop:convexity-multivar}
	Let $S_j \subset \mathbb{R}, ~j=1,2,\cdots,m$ be convex subsets of the real line.
	Let $f:\bigtimes_{j=1}^m S_j\rightarrow\mathbb{R}$ be a function that is of class $\mathcal{C}^0$. % on $\bigtimes_{j=1}^m S_j \subseteq\mathbb{R}^m$.
%	Define $\tr \circ f: \bigtimes_{j=1}^m \mathcal{M}_{\text{diag}^n}(S_j) \rightarrow \mathbb{R}$ as $(\tr\circ f)\left((M)_{l=1}^m\right)= \tr \left( f\left((M)_{l=1}^m\right) \right)$.
%	
%	Then,
%	\begin{itemize}
%		\item[i.] 
%		If $f$ is monotonically increasing (resp. decreasing) with respect to each of its parameters in $\bigtimes_{j=1}^m S_j$, then $\tr \circ f$ is monotonically increasing (resp. decreasing) in $\bigtimes_{j=1}^n \mathcal{M}_{\text{herm}^n}(S_j)$ (\emph{i.e.}, if $M_j\cgeq N_j, ~\forall j\in\{1,2,\cdots,m\}$, then $\tr \circ f((M_l)_{l=1}^m) \geq \tr \circ f((N_l)_{l=1}^m)$).
%		\item[ii.] 

		If $f$ is convex (resp. concave) in $\bigtimes_{j=1}^m S_j$, then $\tr \circ f$ is convex (resp. concave) in $\bigtimes_{j=1}^m \mathcal{M}_{\text{herm}^n}(S_j)$.
		That is, $\tr \left( f \left( \alpha (M_l)_{l=1}^m + (1-\alpha) (N_l)_{l=1}^m \right) \right) \leq \alpha \tr \left( f \left( (M_l)_{l=1}^m \right) \right) + (1-\alpha) \tr \left( f \left( (N_l)_{l=1}^m \right) \right)$ for any $\alpha\in[0,1]$ and $(M_l)_{l=1}^m,$ $(N_l)_{l=1}^m \in \bigtimes_{j=1}^m \mathcal{M}_{\text{herm}^n}(S_j)$ when $f$ is convex (where $\alpha (M_l)_{l=1}^m = (\alpha M_l)_{l=1}^m$).
%	\end{itemize}
\end{proposition}

\begin{proof}~
%	For proof of either of the statements we will consider 
%	Let $(M_l)_{l=1}^m, (N_l)_{l=1}^m \in \bigtimes_{j=1}^m \mathcal{M}_{\text{herm}^n}(S_j)$.
	Let $(\mu_{lk})_{k=1}^n$ and $(\mathbf{m}_{lk})_{k=1}^n$ be eigenvalues and corresponding unit eigenvectors of $M_l$,
	and $(\nu_{lk})_{k=1}^n$ and $(\mathbf{n}_{lk})_{k=1}^n$ be eigenvalues and corresponding unit eigenvectors of $N_l$,
	for $l=1,2,\cdots,m$.
	
%	\vspace{0.5em}\noindent\emph{Proof or `i.':} We prove the statement for monotonically increasing function.
%	\begin{eqnarray*}
%		\tr\left( f\left((M_l)_{l=1}^m\right) \right) - \tr\left( f\left((N_l)_{l=1}^m\right) \right) & = & \displaystyle \sum_{j_1,j_2,\cdots,j_m=1}^n f\left( (\mu_{l j_l})_{l=1}^m \right) ~-~ \sum_{j_1,j_2,\cdots,j_m=1}^n f\left( (\nu_{l j_l})_{l=1}^m \right)
%	\end{eqnarray*}
%	
%	\vspace{0.5em}\noindent\emph{Proof or `ii.':} We prove the statement for a convex $f$.
	Let $\alpha\in[0,1]$.
	Define $Y_l = \alpha M_l + (1-\alpha) N_l, ~l=1,2,\cdots,m$.
	Let $(\lambda_{lk})_{k=1}^n$ and $(\mathbf{y}_{lk})_{k=1}^n$ be eigenvalues and corresponding unit eigenvectors of $Y_l$ for $l=1,2,\cdots,m$. We denote the $p$-th element of $\mathbf{y}_{lk}$ by $[\mathbf{y}_{lk}]_p$.
	
	Since $Y_l$ is Hermitian we can write its $k$-th eigenvalue, $\lambda_{lk} = \mathbf{y}_{lk}^{*} Y_l \mathbf{y}_{lk}$. Thus,
	\begin{eqnarray*}
	& & \tr\left( f\left((Y_l)_{l=1}^m\right) \right) \\
	& = & \displaystyle \sum_{j_1,j_2,\cdots,j_m=1}^n f\left( (\lambda_{l j_l})_{l=1}^m \right) \\
	& = & \displaystyle \sum_{j_1,\cdots,j_m=1}^n f\left( (\mathbf{y}_{l j_l}^{*} Y_l \mathbf{y}_{l j_l})_{l=1}^m \right) \\
	& \leq &  \displaystyle 
				\alpha \sum_{j_1,\cdots,j_m=1}^n f\left( (\mathbf{y}_{l j_l}^{*} M_l \mathbf{y}_{l j_l})_{l=1}^m \right) ~~+~~
				(1-\alpha) \!\!\!\!\sum_{j_1,\cdots,j_m=1}^n f\left( (\mathbf{y}_{l j_l}^{*} N_l \mathbf{y}_{l j_l})_{l=1}^m \right) \\
				& & \qquad\qquad\qquad\qquad\qquad\qquad\qquad\qquad\qquad\qquad\qquad\qquad \text{(Since $f$ is convex.)} \\
	& = & \displaystyle 
				\alpha \sum_{j_1,\cdots,j_m=1}^n f\left( \left( \sum_{k=1}^n |\mathbf{y}_{l j_l}^{*} \mathbf{m}_{l k}|^2 \mu_{lk} \right)_{l=1}^m \right) ~~+~~
				(1-\alpha) \!\!\!\!\sum_{j_1,\cdots,j_m=1}^n f\left( \left( \sum_{k=1}^n |\mathbf{y}_{l j_l}^{*} \mathbf{n}_{l k}|^2 \nu_{lk} \right)_{l=1}^m \right) \\
	& \leq &  \displaystyle 
				\alpha \sum_{j_1,\cdots,j_m=1}^n ~\sum_{k_1,\cdots,k_m=1}^n |\mathbf{y}_{l j_l}^{*} \mathbf{m}_{l k_l}|^2 ~f\left( \left( \mu_{lk_l} \right)_{l=1}^m \right) \\ & & \qquad\qquad\qquad\qquad +~~~
				(1-\alpha) \sum_{j_1,\cdots,j_m=1}^n ~\sum_{k_1,\cdots,k_m=1}^n |\mathbf{y}_{l j_l}^{*} \mathbf{n}_{l k_l}|^2 ~f\left( \left( \nu_{lk_l} \right)_{l=1}^m \right) \\
				& & \qquad\qquad\qquad\qquad\text{(Since $f$ is convex in $\bigtimes_{l=1}^m S_l$, and $\sum_{k=1}^n |\mathbf{y}_{l j_l}^{*} \mathbf{m}_{l k}|^2=1$)} \\
	& = &  \displaystyle 
				\alpha \sum_{k_1,\cdots,k_m=1}^n ~\sum_{j_1,\cdots,j_m=1}^n |\mathbf{y}_{l j_l}^{*} \mathbf{m}_{l k_l}|^2 ~f\left( \left( \mu_{lk_l} \right)_{l=1}^m \right) \\ & & \qquad\qquad\qquad\qquad +~~~
				(1-\alpha) \sum_{k_1,\cdots,k_m=1}^n ~\sum_{j_1,\cdots,j_m=1}^n |\mathbf{y}_{l j_l}^{*} \mathbf{n}_{l k_l}|^2 ~f\left( \left( \nu_{lk_l} \right)_{l=1}^m \right) \\
%				& & \qquad\qquad\qquad\qquad\text{(Since $f$ is convex in $\bigtimes_{l=1}^m S_l$, and $\sum_{k=1}^n |\mathbf{y}_{l j_l}^{*} \mathbf{m}_{l k}|^2=1$)} \\
	& = &  \displaystyle 
				\alpha \sum_{k_1,\cdots,k_m=1}^n ~f\left( \left( \mu_{lk_l} \right)_{l=1}^m \right) 
				~~+~~
				(1-\alpha) \!\!\!\!\sum_{k_1,\cdots,k_m=1}^n ~f\left( \left( \nu_{lk_l} \right)_{l=1}^m \right) \\
				& & \qquad\qquad\qquad\qquad\qquad\qquad\qquad\qquad\qquad\qquad\qquad\text{(Since $\sum_{j=1}^n |\mathbf{y}_{l j}^{*} \mathbf{m}_{l k_l}|^2 = 1$)} \\
	& = &  \displaystyle 
	\alpha\, \tr\left( f\left((M_l)_{l=1}^m\right) \right) ~~+~~ (1-\alpha)\, \tr\left( f\left((N_l)_{l=1}^m\right) \right)
	\end{eqnarray*}
	A similar proof can be given for the concave case.
\end{proof}

\begin{example}
	Consider $f(x,y) = \frac{1}{x^2 y} + \frac{1}{x y^2}$, which is convex in $\mathbb{R}_{\geq 0} \times \mathbb{R}_{\geq 0}$.
	Then the function $\tr\circ f(M,N) = \tr\left( \left( M^2\otimes I \right)^{-1} + \left( I\otimes N^2 \right)^{-1} \right) = \sum_{j=1}^n \sum_{k=1}^n \frac{1}{\mu_j^2 \nu_k} + \frac{1}{\mu_j \nu_k^2}$ (where $\mu_j$ and $\nu_k$ are eigenvalues of $M$ and $N$) is convex in the space of positive semi-definite Hermitian matrices. That is, $f(\alpha M_1 + (1-\alpha)M_2, \alpha N_1 + (1-\alpha)N_2) \leq \alpha f( M_1,N_1) + (1-\alpha) f( M_2,N_2)$, where $M_1, N_1, M_2, N_2$ are positive semi-definite Hermitian
\end{example}

\begin{definition}[Monotonic and Convex Matrix Parameterization]
	Suppose $M:\mathbb{R}\rightarrow \mathcal{M}_{\text{typ}^n}$ describes a matrix of type `$\text{typ}$' parametrized by a real value $t$.
	\begin{itemize}
		\item $M$ is called a monotonically increasing (resp. decreasing) parameterization in $S \subseteq \mathbb{R}$ if $t_2 \geq t_1 ~\Rightarrow~ M(t_2) \cgeq M(t_1)$ (resp. $M(t_2) \cleq M(t_1)$) for all $t_1,t_2\in S$.
		\item $M$ is called a convex (resp. concave) parameterization in $S \subseteq \mathbb{R}$ if $M(\alpha t_1 + (1-\alpha) t_2) \cleq \alpha M(t_1) + (1-\alpha) M(t_2)$ (resp. $\cgeq$) for all $\alpha\in [0,1]$ and $t_1,t_2\in S$.
	\end{itemize}
\end{definition}

\begin{example}
	The matrix parameterization $M(t) = C + t^2 D$, with $D$ positive semi-definite, is monotonially increasing (since $M(t_2) - M(t_1) = (t_2^2 -t_1^2) D$) as well as convex (since $M(\alpha t_1 + (1-\alpha) t_2) = \alpha M(t_1) + (1-\alpha)M(t_2) - \alpha (1-\alpha) (t_2-t_1)^2 D \cleq \alpha M(t_1) + (1-\alpha)M(t_2))$) in $\mathbb{R}_{\geq 0}$.
\end{example}

\begin{corollary} \label{corr:parameterizations}
%	Suppose $f:\mathbb{R}\rightarrow \mathbb{R}$ is a $\mathcal{C}^0$ function on a convex set $S\subseteq \mathbb{R}$.
	Let 
	$f:\bigtimes_{j=1}^m S_j\rightarrow\mathbb{R}$ (where $S_j \subset \mathbb{R}, ~j=1,2,\cdots,m$ are convex subsets of the real line).
	Let $M_l$ be Hermitian matrix parameterizations, parametrized by $t$.
%	 so that 
%	$M_l(t), ~l=1,2,\cdots,m$ are convex matrix parameterizations.
%	we can write $M_l(t) = M_{l0} + M_{l1} t$, where both $M_{l0}$ and $M_{l1}$ are Hermitian and $M_{l1}$ is positive semi-definite for all $l=1,2,\cdots,m$.
	Define $\mathcal{T} = \{t ~|~ M_l(t) \in \mathcal{M}_{\text{herm}^n}(S_l), \forall\,l\in\{1,2,\cdots,m\} \} \subseteq \mathbb{R}$. %(it's easy to check that $\mathcal{T}$ is a convex set).
	Define $g: \mathcal{T} \rightarrow \mathbb{R}$ as $g(t) = \tr\left(f\left((M_l(t))_{l=1}^m\right)\right)$.
	Then,
	\begin{itemize}
		\item[i.] If $f$ is of class $\mathcal{C}^m$ that is monotonically increasing with respect to each of its inputs and if $M_l(t), ~l=1,2,\cdots,m$ are monotonically increasing matrix parameterizations, then $g$ is monotonically increasing.
		\item[ii.] If $f$ is of class $\mathcal{C}^m$ that is monotonically increasing with respect to each of its inputs and is also convex in $\bigtimes_{j=1}^m S_j$, and if $M_l(t), ~l=1,2,\cdots,m$ are convex matrix parameterizations, then $g$ is convex.
		\item[ii.] If $f$ is of class $\mathcal{C}^0$ that is convex in $\bigtimes_{j=1}^m S_j$, and if $M_l(t), ~l=1,2,\cdots,m$ are affine matrix parameterizations, then $g$ is convex.
	\end{itemize}
%	Then,
%	\begin{itemize}
%		\item[i.] If $f$ is monotonically increasing (resp. decreasing) in $S$, then $g$ is monotonically increasing (resp. decreasing) in $\mathcal{T}$.
%		\item[i.] 
%		If $f$ is convex (resp. concave) in $\bigtimes_{l=1}^m S$, then $g$ is convex (resp. concave) in $\mathcal{T}$.
%	\end{itemize}
\end{corollary}

\begin{proof}
	%\todo{$\tr\circ f$ is convex.}
%	Let $(M_l)_{l=1}^m, (N_l)_{l=1}^m \in \bigtimes_{j=1}^m \mathcal{M}_{\text{herm}^n}(S_j)$ and $Y_l = \alpha M_l + (1-\alpha) N_l, ~l=1,2,\cdots,m$.
Let $t_1,t_2 \in\mathcal{T}$ and $\alpha\in[0,1]$.
	
	\vspace{0.5em}\noindent\emph{Proof of `i.':} Since $M_l$ are monotonically increasing matrix parameterizations, for any $t_1 \geq t_2$ we have $M_l(t_1) \cgeq M_l(t_2), \,\forall l=1,2,\cdots,m$.
	Since $f$ is monotonically increasing, using Proposition~\ref{prop:monotonic-multivar-hermitian} we have,
	$\tr\left(f\left( (M_l(t_1))_{l=1}^m \right)\right) \geq \tr\left(f\left( (M_l(t_2))_{l=1}^m \right)\right)$.
	
\vspace{0.5em}\noindent\emph{Proof of `ii.':}
	Since $M_l$ are convex matrix parameterizations, $M_l(\alpha t_1 + (1-\alpha) t_2) ~\cleq~ \alpha M_l(t_1) + (1-\alpha) M_l(t_2)$. Since $f$ is monotonically increasing, using Proposition~\ref{prop:monotonic-multivar-hermitian},
	\begin{eqnarray*}
	\tr\left(f\left( \left(M_l(\alpha t_1 + (1-\alpha)t_2 ) \right)_{l=1}^m\right)\right)
	& \leq &  \tr\left(f\left( \alpha (M_l(t_1))_{l=1}^m + (1-\alpha) (M_l(t_2))_{l=1}^m \right)\right) \\
	& \leq &  \alpha \tr\left(f\left( (M_l(t_1))_{l=1}^m \right)\right) + (1-\alpha) \tr\left(f\left( (M_l(t_2))_{l=1}^m \right)\right) \\ & & \qquad\qquad\qquad\qquad\text{(using Proposition~\ref{prop:convexity-multivar}, since $f$ is convex.)}
	\end{eqnarray*}

\vspace{0.5em}\noindent\emph{Proof of `iii.':}
	Since $M_l$ are affine matrix parameterizations, we can write $M_l(t) = C_l + D_l t$. Thus, $M_l(\alpha t_1 + (1-\alpha)t_2 ) = C_l + D_l (\alpha t_1 + (1-\alpha)t_2) = \alpha M_l(t_1) + (1-\alpha) M_l(t_2)$.
	Thus,
	\begin{eqnarray*}
		\tr\left(f\left( \left(M_l(\alpha t_1 + (1-\alpha)t_2 ) \right)_{l=1}^m\right)\right)
		& = &  \tr\left(f\left( \alpha (M_l(t_1))_{l=1}^m + (1-\alpha) (M_l(t_2))_{l=1}^m \right)\right) \\
		& \leq &  \alpha \tr\left(f\left( (M_l(t_1))_{l=1}^m \right)\right) + (1-\alpha) \tr\left(f\left( (M_l(t_2))_{l=1}^m \right)\right) \\ & & \qquad\qquad\qquad\qquad\text{(using Proposition~\ref{prop:convexity-multivar}, since $f$ is convex.)}
	\end{eqnarray*}
\end{proof}

A special case of the above corollary is affine matrix parameterizations, $M_l(t) = C_l + D_l t$, with $D_l \cgeq 0$. This makes $M_l$ monotonically increasing as well as convex. In this case $g$ is monotomically increasing if $f$ is monotonically increasing with respect to each input, and $g$ is convex if $f$ is convex.

\subsection{Evaluation on the Diagonal}

We now specialize the results from the previous sub-sections to the case when $M_1 = M_2 = \cdots = M_m = L$.
The following proposition is a direct consequence of Propositions~\ref{prop:monotonic-multivar-hermitian} and \ref{prop:convexity-multivar}.
\begin{proposition}
	Let $S \subseteq \mathbb{R}$ be an open, closed or clopen interval of the real line, and $f: S^m \rightarrow\mathbb{R}$.
	Define $F:\mathcal{M}_{\text{herm}^n}(S)\rightarrow \mathbb{R}$ as $F(L) = \tr\left(f(L,L,\cdots,L)\right) = \sum_{j_1,j_2,\cdots,j_m = 1}^n f(\lambda_{j_1}, \lambda_{j_2}, \cdots, \lambda_{j_m})$, where $\lambda_1.\lambda_2,\cdots,\lambda_n$ are the eigenvalues of $L$.
	\begin{itemize}
		\item[i.] If $f$ is of class $\mathcal{C}^m$ and is monotonically increasing (resp. decreasing) with respect to each of its inputs, then $F$ is monotonically increasing (resp. decreasing) in $\mathcal{M}_{\text{herm}^n}(S)$. That is, $L_1 \cgeq L_2 ~\Rightarrow~ F(L_1) \geq F(L_2)$ (resp. $\leq$).
		\item[ii.] If $f$ is of class $\mathcal{C}^0$ and is convex (resp. concave) in $S^m$, then $F$ is convex (resp. concave) in $\mathcal{M}_{\text{herm}^n}(S)$. That is, $F(\alpha L_1 + (1-\alpha)L_2) \,\leq\, \alpha F(L_1) + (1-\alpha) F(L_2)$ (resp. $\geq$)..
	\end{itemize}
\end{proposition}

\section{Application to Function of Graph Spectrum}

In many problems relevant to optimization of graphs with edge weights, one needs to consider a function of the spectrum of the weighted graph Laplacian matrix (for example, see \citep{sahin2024resonance}).
The Laplacian matrix, $L$, of an undirected, simple graph is a real, symmetric positive semi-definite matrix, the $(i,j)$-th element of which is defined as
\[ [L]_{ij} = \left\{ \begin{array}{ll} -w_{ij}, & \text{if $i\neq j$}, \\ \sum_{k=1}^n w_{ik}, & \text{if $i= j$.} \end{array} \right.\]
where $w_{ij}$ is the (real, non-negative) weight on the edge connecting vertices $i$ and vertex $j$ if the edge exists, and zero otherwise.
This allows us to decompose the Laplacian matrix as
\begin{equation*}
L ~=~ \sum_{\{i,j\}\in \mathcal{E}} w_{ij} W_{ij}
\end{equation*}
where $\mathcal{E}$ is the set of (undirected) edges of the graph, and $W_{ij}$ is an $n\times n$ matrix with $1$ at the $i$-th and $j$-th diagonal elements, $-1$ at the $(i,j)$-th and $(j,i)$-th elements, and zero everywhere else.

This makes $L$ affine in each of the edge weights, $w_{ij}$.
Furthermore, it is easy to check that $W_{ij}$ are positive semi-definite of rank $1$ (one eigenvalue equals to $2$, rest are zero).
So, considering $L$ to be parameterized by the edge weight $w_{ij}$, it is a monotonically increasing affine parameterizations.
The following thus follows from Corollary~\ref{corr:parameterizations}.

\begin{proposition}
	Let $L$ be the weighted Laplacian matrix of an undirected, simple graph with non-negative edge weights.
	Consider a function of the form $F(L) = \sum_{k_1,k_2,\cdots,k_m = 1}^n f(\lambda_{k_1}, \lambda_{k_2}, \cdots, \lambda_{k_m})$ (where $\lambda_1.\lambda_2,\cdots,\lambda_n$ are the eigenvalues of $L$ and $f:\mathbb{R}_{\geq 0}^m \rightarrow \mathbb{R}$ is any real-valued function). %Since $F(L)$ is a function of the edge wights, with slight abuse of notation, we interpret $F(L)$ as a function of the edge weights as well.
	\begin{itemize}
		\item[i.] If $f$ is of class $\mathcal{C}^m$ and monotonically increasing with respect to each of its inputs, then $F(L)$ is monotonically increasing with respect to the edge weights.
		\item[ii.] If $f$ is of class $\mathcal{C}^0$ and convex, then $F(L)$ is convex with respect to the edge weights.
	\end{itemize}
\end{proposition}

\newpage
\bibliographystyle{apalike}
\bibliography{refs}

\begin{thebibliography}{}

\bibitem[Abbott et~al., 2001]{abbott2001understanding}
Abbott, S. et~al. (2001).
\newblock {\em Understanding analysis}, volume~2.
\newblock Springer.

\bibitem[Bhatia, 2013]{bhatia}
Bhatia, R. (2013).
\newblock {\em Matrix analysis}, volume 169.
\newblock Springer Science \& Business Media.

\bibitem[Carlen, 2010]{carlen}
Carlen, E. (2010).
\newblock Trace inequalities and quantum entropy: an introductory course.
\newblock {\em Entropy and the quantum}, 529:73--140.

\bibitem[Rahman, 2002]{rahman2002analytic}
Rahman, Q. (2002).
\newblock {\em Analytic theory of polynomials}.
\newblock Oxford University Press.

\bibitem[Sahin et~al., 2024]{sahin2024resonance}
Sahin, A., Kozachuk, N., Blum, R.~S., and Bhattacharya, S. (2024).
\newblock Resonance reduction against adversarial attacks in dynamic networks
  via eigenspectrum optimization.
\newblock {\em arXiv preprint arXiv:2410.00126}.

\end{thebibliography}

\end{document}